%% file: main.tex
\pdfoutput = 1

\documentclass[12pt, a4paper]{amsart}

\input{packages.tex}

\input{commands.tex}

\author{Jan Geuenich}
\address{%
  Fakultät für Mathematik,
  Universität Bielefeld,
  D-33501 Bielefeld%
}
\email{jan.geuenich@math.uni-bielefeld.de}

\title
[%
  Tilting modules for the Auslander algebra of $K\MakeLowercase{[x]/(x^n)}$%
]
{%
  Tilting modules for the \\ Auslander algebra of $K[x]/(x^n)$%
}


\begin{document}

\begin{abstract}
  We construct an isomorphism between the partially ordered set of tilting modules for the Auslander algebra of $K[x]/(x^n)$ and the interval of rational permutation braids in the braid group on $n$ strands.
  Hence, there are only finitely many tilting modules.
\end{abstract}

\maketitle
\thispagestyle{empty}

\section{Introduction}

In this note we classify all tilting modules for the Auslander algebra~$\Lambda_n$ of the truncated polynomial ring $K[x]/(x^n)$.

\medskip

Brüstle, Hille, Ringel and Röhrle \cite{Brustle.Hille.Ringel.ea:99} characterized the classical tilting modules for $\Lambda_n$ as those tilting modules that admit a $\Delta$-filtration with respect to the unique quasi-hereditary structure.
They also parameterized the basic classical tilting modules by the symmetric group~$\Scal_n$, proving there are $c_n = n!$ of them.

Iyama and Zhang \cite{Iyama.Zhang:16} strengthened this result by constructing an anti-isomorphism from~$\Scal_n$ viewed as a poset with the left weak order to the poset of classical tilting modules.
More precisely, they associated with each $w \in \Scal_n$ an ideal~$I_w$ in $\Lambda_n$ that is a classical tilting module.

\medskip

Our main result is the following:

\begin{theorem*}
 [{\cref{corollary:tilting-modules}}]
 The tensor products $I_v \otimes_{\Lambda_n} I_w$ with $v, w \in \Scal_n$ are the basic tilting modules for~$\Lambda_n$.
\end{theorem*}

In order to get a complete and irredundant list of the tilting modules and to determine the tilting poset, we will refine the previous statement.

\medskip

For this, we consider the braid group $\Bcal_n$ with generators $\sfrak_1,\ldots,\sfrak_{n-1}$ subject to the braid relations.
The symmetric group $\Scal_n$ will be regarded as a subset of the braid group $\Bcal_n$ where elements $w \in \Scal_n$ with reduced expression $(i_1,\ldots,i_\ell)$ are identified with $\underline{w} := \sfrak_{i_1} \cdots \sfrak_{i_\ell}$.

The right weak order on $\Scal_n$ then extends to a partial order on $\Bcal_n$ such that $x \leq_R y$ if and only if $x^{-1} y$ can be written as a product of the generators $\sfrak_i$.
The elements $x \in \Bcal_n$ with the property $w_- \leq_R x \leq_R w_+$ now form the interval~\smash{$[w_-,w_+]_R$} of so-called \emph{rational permutation braids} where \smash{$w_- := \underline{w_0}^{-1}$}, \smash{$w_+ := \underline{w_0}$} and $w_0$ is the longest element of \smash{$\Scal_n$}.

\medskip

Observing that the assignment $(v,w) \mapsto \underline{v} \, \underline{w}^{-1}$ defines a bijection between the set of pairs $(v,w) \in \Scal_n \times \Scal_n$ without common right descent and $[w_-,w_+]_R$ (see \cite[Remark~5.9]{Digne.Gobet:17}), we can formulate the refined version of our result:

\begin{theorem*}
 [{\cref{corollary:tilting-poset}}]
 There is a poset isomorphism:
 \[
  \begin{tikzcd}[row sep = 0.1cm]
   [w_-, w_+]_R  \ar[r] & \tilt\Lambda_n
   \\
   \pad[1em]{\underline{v} \, \underline{w}^{-1}} \ar[r, mapsto] & \hspace{1em} I_{w{}^{\phantom{-1}}\hspace{-10pt}} \otimes_{\Lambda_n} I_{v^{-1} w_0} \hspace{-1em}
  \end{tikzcd}
 \]
\end{theorem*}

To illustrate the theorem, we depict below the Hasse diagram of the poset $\tilt\Lambda_3$ where $\otimes = \otimes_{\Lambda_3}$:
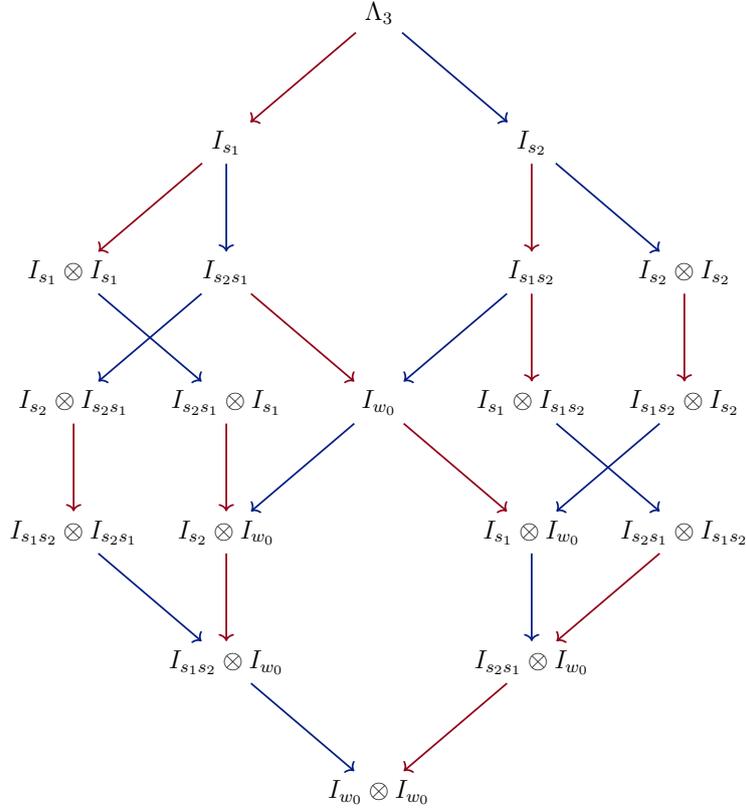
\begin{center}
 \vspace*{-0.2cm}
 \scalebox{0.82}{
  \begin{tikzpicture}[x=3.5cm, y=-3cm, scale=0.7]
  \tikzstyle{m1}=[thick, redish]
  \tikzstyle{m2}=[thick, blueish]

  \node (_) at ( 0, 0) {$\Lambda_3$};

  \node (1) at (-1,1) {$I_{s_1}$};
  \node (2) at ( 1,1) {$I_{s_2}$};

  \node (11) at (-2,2) {$I_{s_1} \otimes I_{s_1}$};
  \node (21) at (-1,2) {$I_{s_2 s_1}$};
  \node (12) at ( 1,2) {$I_{s_1 s_2}$};
  \node (22) at ( 2,2) {$I_{s_2} \otimes I_{s_2}$};

  \node (211) at (-1,3) {$I_{s_2 s_1} \otimes I_{s_1}$};
  \node (221) at (-2,3) {$I_{s_2} \otimes I_{s_2 s_1}$};
  \node (___) at ( 0,3) {$I_{w_0}$};
  \node (112) at ( 1,3) {$I_{s_1} \otimes I_{s_1 s_2}$};
  \node (122) at ( 2,3) {$I_{s_1 s_2} \otimes I_{s_2}$};

  \node (1221) at (-2,4) {$I_{s_1 s_2} \otimes I_{s_2 s_1}$};
  \node (2___) at (-1,4) {$I_{s_2} \otimes I_{w_0}$};
  \node (1___) at ( 1,4) {$I_{s_1} \otimes I_{w_0}$};
  \node (2112) at ( 2,4) {$I_{s_2 s_1} \otimes I_{s_1 s_2}$};

  \node (12___) at (-1,5) {$I_{s_1 s_2} \otimes I_{w_0}$};
  \node (21___) at ( 1,5) {$I_{s_2 s_1} \otimes I_{w_0}$};

  \node (______) at ( 0,6) {$I_{w_0} \otimes I_{w_0}$};

  \draw [->, m1] (_) -- (1);
  \draw [->, m2] (_) -- (2);

  \draw [->, m1] (1) -- (11);
  \draw [->, m2] (1) -- (21);
  \draw [->, m1] (2) -- (12);
  \draw [->, m2] (2) -- (22);

  \draw [->, m2] (11) -- (211);
  \draw [->, m2] (21) -- (221);
  \draw [->, m1] (21) -- (___);
  \draw [->, m2] (12) -- (___);
  \draw [->, m1] (12) -- (112);
  \draw [->, m1] (22) -- (122);

  \draw [->, m1] (211) -- (2___);
  \draw [->, m1] (221) -- (1221);
  \draw [->, m2] (___) -- (2___);
  \draw [->, m1] (___) -- (1___);
  \draw [->, m2] (112) -- (2112);
  \draw [->, m2] (122) -- (1___);

  \draw [->, m1] (2___) -- (12___);
  \draw [->, m2] (1221) -- (12___);
  \draw [->, m1] (2112) -- (21___);
  \draw [->, m2] (1___) -- (21___);

  \draw [->, m2] (12___) -- (______);
  \draw [->, m1] (21___) -- (______);
  \end{tikzpicture}
 }
\end{center}

As a consequence of the previous theorem, the number $t_n$ of isomorphism classes of basic tilting modules for $\Lambda_n$ equals the number of pairs of permutations in~$\Scal_n$ without common right descent.
These integers~$t_n$ form the sequence \href{https://oeis.org/A000275}{OEIS:A000275} and satisfy the recursive formula
\[
  t_0
  \:=\:
  1
  \,,
  \hspace{3em}
  t_n
  \:=\:
  \sum_{k=0}^{n-1} (-1)^{n+k+1} \textstyle\binom{n}{k}^2 \, t_k
  \hspace{1em}
  \text{for $n > 0$.}
\]
The first few values of the sequences $c_n$ and $t_n$ are recorded below:
\begin{center}
 \newcolumntype{D}{C{4em}}
 \scalebox{0.8}{
 \begin{tabular}{l||D|D|D|D|D|D|D}
  $n$ & 1 & 2 & 3 & 4 & 5 & 6 & 7
  \\ \hline\hline
  $c_n$ & 1 & 2 & 6 & 24 & 120 & 720 & 5040
  \\ \hline
  $t_n$ & 1 & 3 & 19 & 211 & 3651 & 90921 & 3081513
 \end{tabular}
 }
\end{center}

\section{Notation and terminology}

For the remaining part of this paper we fix an algebraically closed field $K$ and a finite-dimensional algebra $\Lambda$ over $K$.

\medskip

By a module we always mean a right module.
The category of finite-dimensional $\Lambda$-modules is denoted by $\md\Lambda$ and its bounded derived category by $\Dcal(\Lambda) = \Dcal^b(\md\Lambda)$.
We write $\proj\Lambda$ for the full subcategory of $\md\Lambda$ consisting of projective modules and $\Kcal(\Lambda) = \Kcal^b(\proj\Lambda)$ for its bounded homotopy category.
Given a complex $M \in \Dcal(\Lambda)$ we denote by $\add(M)$ the smallest full additive subcategory and by $\thick(M)$ the smallest full triangulated subcategory of $\Dcal(\Lambda)$ that contains all direct summands of~$M$.
We regard $\Kcal(\Lambda)$ as a full subcategory of~$\Dcal(\Lambda)$.
Objects of $\Dcal(\Lambda)$ isomorphic to objects in $\Kcal(\Lambda)$ are called \emph{perfect}.

\medskip

The \emph{Hasse diagram} of a poset $X$ is the quiver $Q(X)$ with vertex set~$X$ without parallel arrows such that there is an arrow $x \to z$ in~$Q(X)$ if and only if $x > z$ and $x \geq y \geq z$ only for $y \in \{x, z\}$.

\section{Background on tilting theory}

We collect relevant definitions and results from tilting theory needed later.
For details see \cite{Rickard:91,Yekutieli:99,Happel.Unger:05,Happel.Unger:05*1,Aihara.Iyama:12}.

\medskip

A complex $T$ in $\Kcal(\Lambda)$ is \emph{tilting} if $\Hom_{\Kcal(\Lambda)}(T, T[q]) = 0$ for all $q \neq 0$ and $\thick(\Lambda) = \Kcal(\Lambda)$.
A module $T$ in $\md\Lambda$ or a complex $T$ in $\Dcal(\Lambda)$ is said to be \emph{tilting} if $T$ viewed as an object in $\Dcal(\Lambda)$ is isomorphic to a tilting complex in $\Kcal(\Lambda)$.
By a \emph{classical tilting module} over $\Lambda$ we mean a tilting module $T \in \md\Lambda$ with $\projdim T \leq 1$.

\medskip

Tilting complexes $T, T' \in \Dcal(\Lambda)$ are \emph{equivalent} if $\add(T) = \add(T')$.
The set $\tilt^\bullet\Lambda$ of equivalence classes of tilting complexes in $\Dcal(\Lambda)$ forms a poset under the relation
\[
  T \geq T'
  \::\Leftrightarrow\:
  \Hom_{\Dcal(\Lambda)}(T,T'[q]) = 0 \:\: \forall\: q > 0
\]
as shown in \cite[Theorem~2.11]{Aihara.Iyama:12}.
We denote by $\tilt\Lambda$ the subposet of tilting modules and by $\tilt_1\Lambda$ the subposet of classical tilting modules.

\begin{remark}
 \label{remark:tilting-order}
 Set $X^\perp := \{ Y \in \md\Lambda \,|\, \Ext^q_\Lambda(X,Y) = 0 \:\: \forall\: q > 0 \}$.
 Then $T \geq T' \Leftrightarrow T^\perp \supseteq T'^\perp$ for all $T, T' \in \tilt\Lambda$ (see \cite[Lemma~2.1]{Happel.Unger:05*1}).
\end{remark}

Let $\Lambda^e = \Lambda^\op \otimes_K \Lambda$ be the enveloping algebra of $\Lambda$.
A \emph{two-sided tilting complex over $\Lambda$} is a complex $T \in \Dcal(\Lambda^e)$ such that
\[
 T \Lotimes_\Lambda \widetilde{T}
 \:\cong\:
 \Lambda
 \:\cong\:
 \widetilde{T} \Lotimes_\Lambda T
 \hspace{10pt}
 \text{in $\Dcal(\Lambda^e)$}
\]
for some $\widetilde{T} \in \Dcal(\Lambda^e)$.
The set of isomorphism classes of two-sided tilting complexes over $\Lambda$ is known as the \emph{derived Picard group} $\DPic(\Lambda)$.
It is a group with composition $- \Lotimes_\Lambda -$ and identity $\Lambda$ (see \cite{Yekutieli:99}).

\begin{theorem}
 [{\cite{Rickard:91}}]
 \label{theorem:rickard}
 For every two-sided tilting complex $T$ over $\Lambda$ the functor pairs
  \[
    \begin{tikzcd}[column sep = 2cm]
    \Dcal(\Lambda)
    \ar[r, yshift = 0.5ex, "{- \Lotimes_{\Lambda} T}"]
    &
    \Dcal(\Lambda)
    \ar[l, yshift = -0.5ex, "{- \Lotimes_{\Lambda} \widetilde{T}}"]
    \end{tikzcd}
    \hspace{3em}
    \begin{tikzcd}[column sep = 2cm]
    \Dcal(\Lambda^\op)
    \ar[r, yshift = 0.5ex, "{T \Lotimes_{\Lambda} -}"]
    &
    \Dcal(\Lambda^\op)
    \ar[l, yshift = -0.5ex, "{\widetilde{T} \Lotimes_{\Lambda} -}"]
    \end{tikzcd}
  \]
 are quasi-inverse equivalences of triangulated categories.

 \smallskip

 In particular, $T_\Lambda \in \tilt^\bullet \Lambda$ and ${}_\Lambda T \in \tilt^\bullet \Lambda^\op$ are tilting complexes.
\end{theorem}

A complex $T$ in $\Dcal(\Lambda^e)$ is said to be \emph{biperfect} if $T_\Lambda$ is perfect in~$\Dcal(\Lambda)$ and ${}_\Lambda T$ is perfect in $\Dcal(\Lambda^\op)$.
For every $T$ in $\Dcal(\Lambda^e)$ there are canonical morphisms $\Lambda \to \RHom_\Lambda(T,T)$ and $\Lambda \to \RHom_{\Lambda^\op}(T,T)$ in~$\Dcal(\Lambda^e)$, the so-called \emph{left-} and \emph{right-multiplication maps} (see \cite[\S~3]{Yekutieli:92}).

\begin{theorem}
 [{\cite[Proposition~1.8]{Miyachi:03}}]
 \label{theorem:rickard-miyachi}
 Let $T \in \Dcal(\Lambda^e)$ be biperfect.
 The complex $T$ is a two-sided tilting complex over $\Lambda$ if and only if both the left-multiplication map $\Lambda \to \RHom_\Lambda(T,T)$ and the right-multiplication map $\Lambda \to \RHom_{\Lambda^\op}(T,T)$ are isomorphisms in $\Dcal(\Lambda^e)$.
\end{theorem}

The next result by Happel and Unger and its corollary will be crucial.

\begin{theorem}
 [{\cite[Theorem~2.2]{Happel.Unger:05*1}}]
 \label{theorem:happel-unger}
 For all modules $T, T' \in \tilt\Lambda$ with $T' > T$ there exists an arrow $T' \to T''$ in $Q(\tilt\Lambda)$ with $T'' \geq T$.
\end{theorem}

A subquiver $Q'$ of a quiver $Q$ is \emph{successor-closed} if ($v \in Q' \Rightarrow w \in Q'$) for all arrows $v \to w$ in $Q$.

\begin{corollary}
 \label{corollary:finite-successor-closed-tilting-subquiver}
 If $Q'$ is a finite successor-closed subquiver of $Q(\tilt\Lambda)$ with $\Lambda \in Q'$, then $Q' = Q(\tilt\Lambda)$.
\end{corollary}

\begin{proof}
 This follows from the proof of \cite[Corollary~2.2]{Happel.Unger:05}.
\end{proof}

If $\Lambda$ has finite global dimension, the concept of tilting modules coincides with the dual concept of cotilting modules (see \cite[Lemma~1.3]{Happel.Unger:96}).
We continue with some results that are valid in this situation.

\begin{lemma}
 \label{lemma:tilting-cotilting-duality}
 Assume $\gldim\Lambda < \infty$.
 There is a poset isomorphism
 \[
  \begin{tikzcd}[row sep = 0.1cm]
   \tilt\Lambda^\op \ar[r] & (\tilt\Lambda)^\op
   \\
   \pad[1.3em]{T} \ar[r, mapsto] & \pad[0.4em]{D(T)}\hspace{1em}
  \end{tikzcd}
 \]
 induced by the standard duality $D := \Hom_K(-,K) : \md\Lambda^\op \to \md\Lambda$.
\end{lemma}

\begin{proof}
 Indeed, $D(T) \in \tilt\Lambda$ for every $T \in \tilt\Lambda^\op$ because $\Lambda$ has finite global dimension.
 The map clearly is an isomorphism of posets.
\end{proof}

We denote by $\fres\Lambda$ and $\fcores\Lambda$ the posets of functorially finite subcategories of $\md\Lambda$ that are resolving and coresolving, respectively.
The poset structure is given by inclusion (see \cite{Auslander.Reiten:91} for definitions).
All subcategories are assumed to be closed under direct summands.
Moreover, we abbreviate ${}^\perp Y := \{ X \in \md\Lambda \,|\, \Ext^q_\Lambda(X,Y) = 0 \:\: \forall\: q > 0 \}$.
The following characterization of tilting modules is often useful:

\begin{theorem}
[{\cite{Auslander.Reiten:91}, \cite[Corollary~0.3]{Krause.Solberg:03}}]
 \label{theorem:auslander-reiten}
 Assume $\gldim\Lambda < \infty$.
 There is a commutative diagram of poset isomorphisms:
 \[
  \begin{tikzcd}
   &
   \tilt\Lambda
   \ar[dl, "T \mapsto T^\perp"']
   \ar[dr, "T \mapsto {}^\perp (T^\perp)"]
   \\
   \fcores\Lambda
   \ar[rr]
   &&
   (\fres\Lambda)^\op
  \end{tikzcd}
 \]
 Moreover, $\add(T) = T^\perp \cap {}^\perp (T^\perp)$ for every $T \in \tilt\Lambda$.
\end{theorem}

Still assuming $\gldim\Lambda < \infty$, we will end this background section with the remarkable observation that, whenever $\tilt\Lambda$ is finite, every self-orthogonal $\Lambda$-module is partial tilting and $\tilt\Lambda$ forms a lattice.
This is a consequence of the following recent result by Iyama and Zhang:

\begin{theorem}
 [{\cite{Iyama.Zhang:18}}]
 \label{theorem:functorial-finiteness}
 If $\gldim\Lambda < \infty$ and $\tilt\Lambda$ is finite, all resolving and all coresolving subcategories of $\md\Lambda$ are functorially finite.
\end{theorem}

\begin{proof}
 The implication (1)~$\Rightarrow$~(2) of \cite[Theorem~3.2]{Iyama.Zhang:18} and \cite[Corollary~0.3]{Krause.Solberg:03} show that every resolving subcategory is functorially finite.
 Dually, since all cotilting modules are also tilting modules, every coresolving subcategory is functorially finite.
\end{proof}

Recall that $M \in \md\Lambda$ is \emph{self-orthogonal} if $\Ext^q_\Lambda(M,M) = 0$ $\forall\: q > 0$.

\begin{corollary}
 \label{corollary:self-orthogonals-are-partial-tilting}
 If $\gldim\Lambda < \infty$ and $\tilt\Lambda$ is finite, every self-ortho\-gonal module in $\md\Lambda$ is a direct summand of a tilting module.
\end{corollary}

\begin{proof}
 Use \cref{theorem:functorial-finiteness} and the dual of \cite[Proposition~5.12]{Auslander.Reiten:91}.
\end{proof}

\begin{corollary}
 \label{corollary:tilting-lattice}
 If $\gldim\Lambda < \infty$ and $\tilt\Lambda$ is finite, $\tilt\Lambda$ is a lattice.
 In this case, the meet $T \wedge T'$ of elements $T, T'$ in $\tilt\Lambda$ is given by
 \[
  (T \wedge T')^\perp
  \:=\:
  T^\perp \cap T'^\perp
  \,.
 \]
\end{corollary}

\begin{proof}
 \Cref{theorem:functorial-finiteness} shows that $\fres\Lambda$ admits meets, since the intersection of resolving subcategories is resolving.
 Analogously, $\fcores\Lambda$ admits meets.
 Now apply \cref{theorem:auslander-reiten}.
\end{proof}

\section{The Auslander algebra\texorpdfstring{ of $K[x]/(x^n)$}{}}

From now on let $\Lambda = \Lambda_n$ be the path algebra over $K$ of the quiver
\[
 \begin{tikzcd}[column sep = 2cm]
  1
  \ar[r, yshift=-0.5ex, "\alpha_1"']
  &
  2
  \ar[l, yshift=0.5ex, "\beta_2"']
  \ar[r, yshift=-0.5ex, "\alpha_2"']
  &
  \ar[l, yshift=0.5ex, "\beta_3"']
  \ar[r, dash, dotted]
  &
  \ar[r, yshift=-0.5ex, "\alpha_{n-1}"']
  &
  n
  \ar[l, yshift=0.5ex, "\beta_{n\phantom{-1}}"']
 \end{tikzcd}
\]
modulo the relations $\alpha_1 \beta_2$ and $\alpha_i \beta_{i+1} - \beta_i \alpha_{i-1}$ for all $1 < i < n$.

\medskip

Then $\Lambda$ is the Auslander algebra of $K[x]/(x^n)$, i.e.\ the endomorphism algebra of the direct sum of the $n$ indecomposable $K[x]/(x^n)$-modules $K[x]/(x^i)$ with $1 \leq i \leq n$.
Its classical tilting, support $\tau$-tilting and exceptional modules were investigated in \cite{Brustle.Hille.Ringel.ea:99,Iyama.Zhang:16,Hille.Ploog:17}.

\begin{remark}
 Some basic properties of $\Lambda$ are collected in \cite[\S~1]{Hille.Ploog:17}.
 For us it is important to know that $\gldim\Lambda \leq 2$ and $e_n\Lambda$ is a projective-injective $\Lambda$-module.
\end{remark}

\section{Classical tilting modules}

The classical tilting modules for the Auslander algebra $\Lambda = \Lambda_n$ are classified in \cite{Brustle.Hille.Ringel.ea:99}.
An explicit anti-isomorphism between the poset $\tilt_1 \Lambda$ and the symmetric group $\Scal = \Scal_n$ on $n$ letters with the left weak order was established by Yusuke Tsujioka in his Master's thesis.
We will recall this classification as presented in \cite{Iyama.Zhang:16} below.

\medskip

For $1 \leq i < n$ we denote by $s_i \in \Scal$ the trans\-position of $i$ and $i+1$.
The \emph{length} of $w \in \Scal$ is
\[
  \ell(w)
  \::=\:
  \sharp\{(i,j) \,|\, 1 \leq i < j \leq n, w(i) > w(j) \}
  \,.
\]
A sequence $(i_1,\ldots,i_\ell)$ in $\{1,\ldots,n-1\}$ is said to be a \emph{reduced expression} for $w$ if $w = s_{i_1} \cdots s_{i_\ell}$ and $\ell = \ell(w)$.
The \emph{left weak order} $\geq_L$ and the \emph{right weak order} $\geq_R$ on $\Scal$ are defined by:
\[
  \begin{array}{lcl}
   w \geq_L v
   &:\Leftrightarrow&
   \ell(w) = \ell(v) + \ell(wv^{-1})
   \\ [0.2em]
   w \geq_R v
   &:\Leftrightarrow&
   \ell(w) = \ell(v) + \ell(v^{-1}w)
  \end{array}
\]
Each of these two orders turns $\Scal$ into a lattice with maximal element
\[
  w_0
  \::=\:
  \left(\begin{smallmatrix*}
   1 & 2 & \cdots & n-1 & n
   \\
   n & n-1 & \cdots & 2 & 1
  \end{smallmatrix*}\right)
  \:\in\:
  \Scal
  \,.
\]
Observe that the arrows in the Hasse diagram $Q(\Scal, \geq_L)$ are $s_i v \to v$ for $v \in \Scal$ and $1 \leq i < n$ with $s_i v >_L v$ and the assignment $w \mapsto w^{-1}$ yields an isomorphism of posets $(\Scal, \geq_R) \to (\Scal, \geq_L)$.

\medskip

For $1 \leq i < n$ let $I_i$ be the ideal $\Lambda (1-e_i) \Lambda$ in $\Lambda$.
More generally, we define for all $w \in \Scal$
\[
 I_w
 \::=\:
 I_{i_1} \cdots I_{i_\ell}
\]
where $\underline{i} = (i_1, \ldots, i_\ell)$ is any reduced expression for $w$.
For a proof why the ideal $I_w$ in $\Lambda$ only depends on $w$ and not on the particular choice of the reduced expression $\underline{i}$ see \cite[Proposition~3.15]{Iyama.Zhang:16}.

\medskip

\begin{theorem}
 [{\cite[Theorem~3.18]{Iyama.Zhang:16}}]
 \label{theorem:classical-tilting-poset}
 There are poset isomorphisms:
 \[
  \begin{tikzcd}[row sep = 0.1cm]
   (\Scal, \geq_L) \ar[r] & (\tilt_1 \Lambda)^\op
   &
   (\Scal, \geq_R) \ar[r] & (\tilt_1 \Lambda^\op)^\op
   \\
   \pad[1.5em]{w} \ar[r, mapsto] & \pad[1em]{I_w}\hspace{1.5em}
   &
   \pad[1.5em]{w} \ar[r, mapsto] & \pad[1.4em]{I_w}\hspace{1.5em}
  \end{tikzcd}
 \]
\end{theorem}

\begin{theorem}
 [{\cite[Theorem~3.5]{Iyama.Zhang:16}}]
 \label{theorem:endomorphism-algebra-of-classical-tilting-modules}
 For every element $w \in \Scal$ both the left-multiplication map $\Lambda \to \End_\Lambda(I_w)$ and the right-multiplication map $\Lambda^\op \to \End_{\Lambda^\op}(I_w)$ are isomorphisms of algebras.
\end{theorem}

\section{The derived Picard group}

The goal of this section is to construct a homomorphism from the braid group $\Bcal = \Bcal_n$ on $n$ strands to the derived Picard group $\DPic(\Lambda)$.

\begin{proposition}
 For every element $w \in \Scal$ the ideal $I_w$ is a two-sided tilting complex over $\Lambda$.
\end{proposition}

\begin{proof}
 It is $\RHom_\Lambda(T, T) = \End_\Lambda(T)$ and $\RHom_{\Lambda^\op}(T, T) = \End_{\Lambda^\op}(T)$ for $T = I_w$ by \cref{theorem:classical-tilting-poset}.
 Now use \cref{theorem:endomorphism-algebra-of-classical-tilting-modules,theorem:rickard-miyachi}.
\end{proof}

Before stating this section's main result, we recall properties of braid groups.
Details can be found for example in \cite[Chapter~6]{Kassel.Turaev:08}.

\medskip

Let $\Fcal = \Fcal_n$ be the free group with generators $\sfrak_i$ for $1 \leq i < n$.
The braid group $\Bcal$ is by definition the quotient of $\Fcal$ by the relations
\[
 \label{braid-relations}
 \tag{$\star$}
 \hspace{2em}
 \arraycolsep 2pt
 \begin{array}{rcll}
  \sfrak_i \sfrak_j
  &=&
  \sfrak_j \sfrak_i
  &
  \hspace{2em}\text{for $1 \leq i < i+1 < j < n$,}
  \\
  \sfrak_i \sfrak_{i+1} \sfrak_i
  &=&
  \sfrak_{i+1} \sfrak_i \sfrak_{i+1}
  &
  \hspace{2em}\text{for $1 \leq i < i+1 < n$.}
 \end{array}
\]

We denote by $\Bcal_+$ the monoid generated by $\sfrak_i$ with $1 \leq i < n$ subject to the same relations~\cref{braid-relations}.
It is well-known that the canonical monoid morphism $\Bcal_+ \to \Bcal$ is injective.
In this way, $\Bcal_+$ becomes a subset of $\Bcal$.
With $\Bcal_- := (\Bcal_+)^{-1}$ we then have
\[
 \Bcal
 \:=\:
 \Bcal_+ \Bcal_-
 \:=\:
 \Bcal_- \Bcal_+
 \,.
\]
The \emph{length} of an element $x = \sfrak_{i_1} \cdots \sfrak_{i_\ell}$ in~$\Bcal_+$ is the integer $\ell(x) = \ell$.

\medskip

The rule $\sfrak_i \mapsto s_i$ induces a surjective group morphism $\Bcal \to \Scal, x \mapsto \overline{x}$, whose kernel is generated by the elements \smash{$\sfrak_i^2$}.
Furthermore, we can and will regard the symmetric group~$\Scal$ as a subset~$\Scal_+$ of $\Bcal_+$ by identifying each element $w = s_{i_1} \cdots s_{i_{\ell}} \in \Scal$ of length~$\ell$ with $\underline{w} := \sfrak_{i_1} \cdots \sfrak_{i_\ell} \in \Bcal_+$.
We write~$w_+$ for $w_0 \in \Scal$ when considered as the element $\underline{w_0}$ of $\Scal_+$.

\begin{remark}
 The pair $(\Bcal_+, w_+)$ is a \emph{comprehensive Garside monoid} in the sense of \cite[Theorem~6.20]{Kassel.Turaev:08} and the canonical map $\Bcal_+ \to \Bcal$ is the embedding into its group of fractions.
\end{remark}

After the preparation, we get to the promised result:

\begin{proposition}
 \label{proposition:braid-group-to-derived-picard-group}
 The map $\Scal \to \DPic(\Lambda)$ given by $w \mapsto I_w$ extends to a group homomorphism:
 \[
  \begin{tikzcd}[row sep = 0.1cm]
   \Bcal \ar[r] & \DPic(\Lambda)
   \\
   \pad[0.5em]{x} \ar[r, mapsto] & \pad[1.5em]{T_x}
  \end{tikzcd}
 \]
\end{proposition}

\begin{proof}
 Consider the diagram
 \[
  \begin{tikzcd}[]
   &
   \Fcal
   \ar[ld, "\pi"']
   \ar[rd, "\varphi"]
   \\
   \Bcal
   \ar[rr, dotted]
   &&
   \DPic(\Lambda)
  \end{tikzcd}
 \]
 where $\pi$ and $\varphi$ are the group morphisms given by $\sfrak_i \mapsto \sfrak_i$ and $\sfrak_i \mapsto I_i$, respectively.
 For all reduced expressions $(i_1,\ldots,i_\ell)$ for $w \in \Scal$ we have by \cite[Propositions~3.17]{Iyama.Zhang:16} in $\DPic(\Lambda)$
 \[
  \varphi\big(\sfrak_{i_1} \cdots \sfrak_{i_\ell}\big)
  \:=\:
  I_{i_1} \Lotimes_\Lambda \cdots \Lotimes_\Lambda I_{i_\ell}
  \:=\:
  I_w
  \,.
 \]
 It follows that $\varphi$ factors over $\pi$ because $\Bcal$ is defined by relations $v = w$ with \smash{$v = \sfrak_{i_1} \cdots \sfrak_{i_\ell}, w = \sfrak_{j_1} \cdots \sfrak_{j_\ell} \in \Fcal$} where $(i_1,\ldots,i_\ell)$ and $(j_1,\ldots,j_\ell)$ are reduced expressions for the element \smash{$\overline{\pi(v)} = \overline{\pi(w)} \in \Scal$}.
\end{proof}

\section{Tilting complexes}

Composing the map $\Bcal_+ \to \DPic(\Lambda)$ from \cref{proposition:braid-group-to-derived-picard-group} with the canonical map $\DPic(\Lambda) \to \tilt^\bullet\Lambda$ yields a map $\Bcal_+ \to \tilt^\bullet\Lambda$.
In this section we discuss why this map becomes an anti-morphism of posets when endowing $\Bcal_+$ with the right divisibility order.
Furthermore, we show that it preserves covering relations.

\medskip

The \emph{right-divisibility order}~$\geq_L$ and the \emph{left-divisibility order}~$\geq_R$ are extensions of $\geq_L$ and $\geq_R$ from $\Scal_+$ to $\Bcal$ where for $v, w \in \Bcal$:
\[
  \begin{array}{lcl}
   y \geq_L x
   &:\Leftrightarrow&
   y x^{-1} \in \Bcal_+
   \\ [0.2em]
   y \geq_R x
   &:\Leftrightarrow&
   x^{-1} y \in \Bcal_+
  \end{array}
\]

\begin{proposition}
 \label{proposition:braid-monoid-to-tilting-complexes}
 There is a morphism of strict posets:
 \[
  \begin{tikzcd}[row sep = 0.1cm]
   (\Bcal_+, >_L) \ar[r] & (\tilt^\bullet\Lambda)^\op
   \\
   \pad[1.8em]{x} \ar[r, mapsto] & \pad[1em]{T_x}\hspace{1em}
  \end{tikzcd}
 \]
\end{proposition}

\begin{proof}
 It suffices to verify $T_x > T_{\sfrak_i x}$ for every $x \in \Bcal_+$ and $1 \leq i < n$.
 \Cref{theorem:classical-tilting-poset} shows $\Lambda > I_i$ and we get $T_x = \Lambda \Lotimes_\Lambda T_x > I_i \Lotimes_\Lambda T_x = T_{\sfrak_i x}$ in $\tilt^\bullet\Lambda$ with \cref{theorem:rickard,proposition:braid-group-to-derived-picard-group}.
\end{proof}

The following fact is a variation of \cite[Lemma~4.3]{Iyama.Zhang:16}:

\begin{lemma}
 \label{lemma:approximations}
 For all $1 \leq i < n$  we have a short exact sequence
 \[
  0
  \to
  e_i \Lambda
  \xrightarrow{\pad[3pt]{\iota \,=\, \left(\begin{smallmatrix} \alpha_{i-1} \cdot \\ \beta_{i+1} \cdot \end{smallmatrix}\right)}}
  e_{i-1} \Lambda \oplus e_{i+1} \Lambda
  \xrightarrow{\pi \,=\, \left(\begin{smallmatrix} -\beta_i \cdot & \alpha_i \cdot \end{smallmatrix}\right)}
  e_i I_i
  \to
  0
 \]
 in $\md\Lambda$ where $\iota$ is a minimal left and $\pi$ a minimal right $\add((1-e_i)\Lambda)$-approximation and by convention $e_0 := 0$.
\end{lemma}

\begin{proof}
 This short exact sequence is the minimal projective resolution of $e_i I_i = \rad(e_i \Lambda)$.
 For $j \neq i$, applying $\Hom_\Lambda(-, e_j \Lambda)$ and $\Hom_\Lambda(e_j \Lambda, -)$ yields exact sequences:
 \[
  \begin{tikzcd}[row sep = 1ex, column sep = 3ex]
   \Hom_\Lambda(e_{i-1} \Lambda \oplus e_{i+1} \Lambda, e_j \Lambda)
   \ar[r, "\iota^*"]
   &
   \Hom_\Lambda(e_i \Lambda, e_j \Lambda)
   \ar[r]
   &
   \Ext^1_\Lambda(e_i I_i, e_j \Lambda)
   \\
   \Hom_\Lambda(e_j \Lambda, e_{i-1} \Lambda \oplus e_{i+1} \Lambda)
   \ar[r, "\pi_*"]
   &
   \Hom_\Lambda(e_j \Lambda, e_i I_i)
   \ar[r]
   &
   \Ext^1_\Lambda(e_j \Lambda, e_i \Lambda)
  \end{tikzcd}
 \]
 Letting $S_i$ be the simple $\Lambda^e$-module given by the short exact sequence $0 \to I_i \to \Lambda \to S_i \to 0$ we have with \cite[Lemma~3.6]{Iyama.Zhang:16}
 \[
  \Ext^1_\Lambda(e_i I_i, e_j \Lambda)
  \:\cong\:
  \Ext^2_\Lambda(S_i, e_j \Lambda)
  \:\cong\:
  e_j \Lambda \otimes_\Lambda S_i
  \:=\:
  0
  \,.
 \]
 Clearly, $\Ext^1_\Lambda(e_j \Lambda, e_i \Lambda) = 0$, too.
 Therefore $\iota$ and $\pi$ are $\add((1-e_i)\Lambda)$-approximations.
 Both of them are minimal by \cite[Proposition~1.1]{Auslander.Reiten:91}, since neither $e_i I_i$ nor $e_i \Lambda$ is a direct summand of $e_{i-1} \Lambda \oplus e_{i+1} \Lambda$.
\end{proof}

The next lemma will be essential to determine $Q(\tilt\Lambda)$.

\begin{lemma}
 \label{lemma:mutation-triangle}
 For all $x \in \Bcal_+$ and $1 \leq i < n$ we have a triangle
 \[
  \begin{tikzcd}
   e_i T_x
   \ar[r, "\iota"]
   &
   e_{i-1} T_x \oplus e_{i+1} T_x
   \ar[r, "\pi"]
   &
   e_i T_{\sfrak_i x}
   \ar[r]
   &
   \cdot
  \end{tikzcd}
 \]
 in $\Dcal(\Lambda)$ where $\iota$ is a minimal left and $\pi$ a minimal right $\add((1-e_i)T_x)$-approximation.
 Furthermore, $e_j T_{\sfrak_i x} = e_j T_x$ for all $1 \leq j \leq n$ with $j \neq i$.
\end{lemma}

\begin{proof}
Apply $- \Lotimes_\Lambda T_x$ to the triangle $e_i \Lambda \to e_{i-1} \Lambda \oplus e_{i+1} \Lambda \to e_i I_i \to \cdot \:$ induced by the sequence in \cref{lemma:approximations} and to the identities $e_j I_i = e_j \Lambda$.
Then use \cref{theorem:rickard}.
\end{proof}

\begin{corollary}
 \label{corollary:arrows-in-hasse-diagram}
 There is an arrow $T_x \xrightarrow{\pad[1ex]{}} T_{\sfrak_i x}$ in the quiver $Q(\tilt^\bullet \Lambda)$ for all $x \in \Bcal_+$ and $1 \leq i < n$. \end{corollary}

\begin{proof}
 Use \cref{lemma:mutation-triangle} and \cite[Theorem~2.35]{Aihara.Iyama:12}.
\end{proof}

For $x \in \Bcal_+$ and $1 \leq i \leq n$ it makes sense to refer to $e_i T_x$ as the \emph{$i$-th summand} of $T_x$, since by \cref{theorem:rickard}
\[
 \dim_K \End_{\Dcal(\Lambda)}(e_i T_x)
 \:=\:
 \dim_K \End_\Lambda(e_i \Lambda)
 \:=\:
 i
 \,.
\]
We write \smash{$T_x \xrightarrow{\pad[0.5ex]{i}} T_{\sfrak_i x}$} for an arrow $T_x \xrightarrow{\pad[0.5ex]{\phantom{i}}} T_{\sfrak_i x}$ in $Q(\tilt^\bullet \Lambda)$ to emphasize the fact that it corresponds to mutating the $i$-th summand.

\medskip

We close this section with an interesting observation that will enable us to determine the possible dimension vectors of tilting modules for $\Lambda$.
For this purpose, let $V = K_0(\Lambda)$ be the Grothendieck group of $\Dcal(\Lambda)$.
The symmetric group $\Scal$ acts on $V^n$ via
\[
  s_i \cdot (\ldots, v_{i-1}, v_i, v_{i+1}, \ldots)
  \:=\:
  (\ldots, v_{i-1}, v_{i-1} - v_i + v_{i+1}, v_{i+1}, \ldots)
\]
with $v_0 := 0$.
For each $x \in \Bcal_+$ we define $d(T_x)$ as the element in $V^n$ whose $i$-th component is the equivalence class of the $i$-th summand $e_i T_x$.
It can be computed by the following formula:

\begin{lemma}
 $d(T_x) = \overline{x} \cdot d(\Lambda)$ for all $x \in \Bcal_+$.
\end{lemma}

\begin{proof}
 The case $x = 1$ is trivial.
 Otherwise we write $x = \sfrak_i y$ for some $i$ and $y \in \Bcal_+$.
 Let $u := d(T_x)$.
 By induction we have $v := d(T_y) = \overline{y} \cdot d(\Lambda)$.
 With \cref{lemma:mutation-triangle} we get $u_i = v_{i-1} - v_i + v_{i+1}$ and $u_j = v_j$ for all $j \neq i$.
 We conclude $u = s_i \cdot v = \overline{\sfrak_i} \cdot (\overline{y} \cdot d(\Lambda)) = \overline{x} \cdot d(\Lambda)$.
\end{proof}

\begin{corollary}
 \label{corollary:tilting-summands-in-grothendieck-group}
 The set $\{ d(T_x) \,|\, x \in \Bcal_+ \} = \{ d(I_w) \,|\, w \in \Scal \}$ is finite.
\end{corollary}

\section{Tilting modules}

In this section we finally classify the tilting modules for the Auslander algebra $\Lambda$ and determine the poset structure of $\tilt\Lambda$.
We begin with four lemmata that serve as the main steps of the classification's proof.

\begin{lemma}
 \label{lemma:product-of-classical-tilting-is-tilting}
 $T_{\underline{v} \, \underline{w}} = I_v \Lotimes_\Lambda I_w \cong I_v \otimes_\Lambda I_w \in \tilt\Lambda$ for all $v, w \in \Scal$.
\end{lemma}

\begin{proof}
 The short exact sequence $0 \to I_w \to \Lambda \to \Lambda/I_w \to 0$ shows that $\Tor^\Lambda_q(I_v, I_w) \cong \Tor^\Lambda_{q+1}(I_v, \Lambda/I_w)$ for all $q > 0$.
 Since $\projdim\: (I_v)_\Lambda \leq 1$, we see that \smash{$I_v \Lotimes_\Lambda I_w \cong I_v \otimes_\Lambda I_w$} is a module.
\end{proof}

We use the notation $[a,b]_L$ for the interval $\{x \in \Bcal \,|\, a \leq_L x \leq_L b \}$ and define the interval $[a,b]_R$ similarly.
Let $\Scal_- := (\Scal_+)^{-1}$ and $w_- := (w_+)^{-1}$.
The elements of
\[
 [w_-, w_+]
 \::=\:
 [w_-, w_+]_L
 \:=\:
 [w_-, w_+]_R
 \:=\:
 \Scal_+ \Scal_-
 \:=\:
 \Scal_- \Scal_+
\]
are the \emph{rational permutation braids} studied in \cite[Proposition~4.3]{Digne.Gobet:17}.

\medskip

It is not hard to see that for $x \in [w_-, w_+]$ the module $T_{w_+ x} \in \tilt\Lambda^\op$ corresponds to $T_{x^{-1} w_+} \in \tilt\Lambda$ under the isomorphism from \cref{lemma:tilting-cotilting-duality}.
We prove a special case:

\begin{lemma}
 \label{lemma:injective-cogenerator-is-square-of-longest-element}
 $T_{w_+^2} = D(\Lambda)$ in $\tilt\Lambda$.
\end{lemma}

\begin{proof}
 Let $(i_1,\ldots,i_\ell)$ be a reduced expression for $w_0$.
 By \cref{theorem:classical-tilting-poset} there are paths
 \[
  \begin{tikzcd}[
  	row sep = 1ex,
    /tikz/column 4/.append style={anchor=base west}
   ]
   \Lambda_\Lambda
   \ar[r, "i_\ell"]
   &
   \cdots
   \ar[r, "i_1"]
   &
   (I_{w_0})_\Lambda
   &
   \text{in $Q(\tilt\Lambda)$ and}
   \\
   {}_\Lambda \Lambda
   \ar[r, "i_1"]
   &
   \cdots
   \ar[r, "i_\ell"]
   &
   {}_\Lambda (I_{w_0})
   &
   \text{in $Q(\tilt\Lambda^\op)$.}
  \end{tikzcd}
 \]
 Now, $(I_{w_0})_\Lambda$ is the unique module $T_\Lambda \in \tilt\Lambda$ with $\projdim\: T_\Lambda \leq 1$ and $\injdim\: T_\Lambda \leq 1$ (see \cite{Brustle.Hille.Ringel.ea:99}).
 Similarly, ${}_\Lambda (I_{w_0})$ is the unique module ${}_\Lambda T \in \tilt\Lambda^\op$ with $\projdim\: {}_\Lambda T \leq 1$ and $\injdim\: {}_\Lambda T \leq 1$.
 Because of \cref{lemma:tilting-cotilting-duality} we must have $(I_{w_0})_\Lambda \cong D({}_\Lambda (I_{w_0}))$ such that there is a path
 \[
  \begin{tikzcd}[row sep = 1ex]
   \Lambda
   \ar[r, "i_\ell"]
   &
   \cdots
   \ar[r, "i_1"]
   &
   I_{w_0}
   \ar[r, "i_\ell"]
   &
   \cdots
   \ar[r, "i_1"]
   &
   D(\Lambda)
  \end{tikzcd}
 \]
 in $Q(\tilt\Lambda)$.
 \Cref{corollary:arrows-in-hasse-diagram} now implies $T_{w_+^2} = D(\Lambda)$ in $\tilt\Lambda$.
\end{proof}

The next insight is a consequence of Voigt's lemma.

\begin{lemma}
 \label{lemma:only-finitely-many-tilting-modules}
 The set $\mathbb{T} = \{ T_x \,|\, x \in \Bcal_+ \} \cap \tilt\Lambda$ is finite.
\end{lemma}

\begin{proof}
 Let $X = \{ \dim_K T_x \,|\, T_x \in \mathbb{T} \}$.
 On the one hand, for each $d \in X$ the set $\{ T_x \in \mathbb{T} \,|\, \dim_K T_x = d \}$ is finite because of \cite[Corollary~9]{Huisgen-Zimmermann.Saorin:01}.
 On the other hand, \cref{corollary:tilting-summands-in-grothendieck-group} implies $X \subseteq \{ \dim_K I_w \,|\, w \in \Scal \}$, so $X$ is finite, too.
 This proves the claim.
\end{proof}

We formulate one last lemma before turning to the classification.

\begin{lemma}
 \label{lemma:tilting-quiver}
 Let $Q'$ be the full subquiver of $Q = Q(\tilt\Lambda)$ spanned by~$\mathbb{T}$.
 Then $Q' = Q$ and every arrow in this quiver is of the form \smash{$T_x \xrightarrow{\pad[2pt]{i}} T_{\sfrak_i x}$} for some $x \in \Bcal_+$ and $1 \leq i < n$.
\end{lemma}

\begin{proof}
 We show that $Q'$ is a successor-closed subquiver of $Q$ so that, using \cref{lemma:only-finitely-many-tilting-modules}, \cref{corollary:finite-successor-closed-tilting-subquiver} is applicable:
 Let $T_x \to T$ be an arrow in~$Q$ for some $x \in \Bcal_+$.
 According to \cite[\S~1]{Happel.Unger:05*1}, there exist $1 \leq i \leq n$, an indecomposable $\Lambda$-module $Y$ such that $T = (1-e_i) T_x \oplus Y$ and a short exact sequence
 \smash{$
  0
  \to
  e_i T_x
  \xrightarrow{\iota}
  E
  \to
  Y
  \to
  0
 $}
 in which $\iota$ is a minimal left $\add((1-e_i) T_x)$-approximation.
 Given that the projective-injective module~$e_n \Lambda$ appears as a summand of every tilting module, we have $e_n T_x \cong e_n \Lambda$, so $i \neq n$.
 Thus $Y \cong e_i T_{\sfrak_i x}$ and $T = T_{\sfrak_i x}$ by \cref{lemma:mutation-triangle}.
\end{proof}

Now we are ready to prove our main result.

\begin{theorem}
 \label{theorem:tilting-poset}
 There is a poset isomorphism:
 \[
  \begin{tikzcd}[row sep = 0.1cm]
   [1, w_+^2]_L  \ar[r] & (\tilt\Lambda)^\op
   \\
   \pad[1em]{x}\hspace{1em} \ar[r, mapsto] & \pad[1em]{T_x}\hspace{1em}
  \end{tikzcd}
 \]
\end{theorem}

\begin{proof}
 The map is well-defined by \cref{lemma:product-of-classical-tilting-is-tilting} because $[1, w_+^2]_L = \Scal_+ \Scal_+$.
 According to \cref{proposition:braid-monoid-to-tilting-complexes} it is a morphism of posets.

 \medskip
 (1)
 \emph{For all $T_x, T_y \in \tilt\Lambda$ with $x, y \in \Bcal_+$ and $T_x \geq T_y$ there is $z \in \Bcal_+$ with $T_{zx} = T_y$:}
 Given a path \smash{$T_x = T_{x_0} \xrightarrow{} \cdots \xrightarrow{} T_{x_\ell}$} in $Q$ with $T_{x_\ell} \geq T_y$ and $x_k = \sfrak_{i_k} \cdots \sfrak_{i_1} x$ for all $0 \leq k \leq \ell$, either $T_{x_\ell} = T_y$ or by \cref{theorem:happel-unger,lemma:tilting-quiver} there is an arrow \smash{$T_{x_\ell} \xrightarrow{} T_{x_{\ell+1}}$} in~$Q$ with $T_{x_{\ell+1}} \geq T_y$ and $x_{\ell+1} = \sfrak_{i_{\ell+1}} x_{\ell}$ for some $i_{\ell+1}$.
 If our claim were false, we would get an infinite path \smash{$T_{x_0} \xrightarrow{} \cdots \xrightarrow{} T_{x_\ell} \xrightarrow{} \cdots$} in contradiction to \cref{lemma:only-finitely-many-tilting-modules}.

 \medskip
 (2)
 \emph{For all $T \in \tilt\Lambda$ and $x, y \in \Bcal_+$ with $T_x = T = T_y$ we have $x = y$:}
 Our argument uses induction on $\ell = \min\{\ell(x), \ell(y)\}$ and follows \cite[Lemma~6.4]{Aihara.Mizuno:17}.
 If $\ell = 0$, we have $T = \Lambda$, so $x = 1 = y$ by \cref{proposition:braid-monoid-to-tilting-complexes}.
 Otherwise we can write $x = x' \sfrak_i$, $y = y' \sfrak_j$ for some $i, j$ and $x', y' \in \Bcal_+$.
 Let $\sfrak_{ij}$ be the join of the elements $\sfrak_i$ and $\sfrak_j$ in the lattice $(\Scal_+, \geq_L)$, i.e.
 \[
  \sfrak_{ij}
  \:=\:
  \left\{
  \begin{array}{cl}
    \sfrak_i = \sfrak_j
    &
    \text{if $i = j$,}
    \\
    \sfrak_i \sfrak_j = \sfrak_j \sfrak_i
    &
    \text{if $|i - j| > 1$,}
    \\
    \sfrak_i \sfrak_j \sfrak_i = \sfrak_j \sfrak_i \sfrak_j
    &
    \text{if $|i - j| = 1$.}
  \end{array}
  \right.
 \]
 Then \smash{$T_{\sfrak_{ij}}^\perp = T_{\sfrak_i}^\perp \cap T_{\sfrak_j}^\perp$} with \cite[Theorem~4.12]{Iyama.Zhang:16}, \cite[Remark~1.13]{Iyama.Reiten.Thomas.ea:15} and  \cite{Adachi.Iyama.Reiten:14}.
 Now $T_{\sfrak_i} \geq T$ and $T_{\sfrak_j} \geq T$ because of $x \geq_L \sfrak_i$ and $y \geq_L \sfrak_j$.
 Hence, $T_{\sfrak_{ij}} \geq T$ by \cref{remark:tilting-order}.
 So by (1) there is $z \in \Bcal_+$ with $T_{z \sfrak_{ij}} = T$.
 Consequently, $T_{z_i} = T_{x'}$ and $T_{z_j} = T_{y'}$ for \smash{$z_i = z \sfrak_{ij} \sfrak_i^{-1}$} and \smash{$z_j = z \sfrak_{ij} \sfrak_j^{-1}$} by \cref{proposition:braid-group-to-derived-picard-group}.
 Without loss of generality we may assume \smash{$\ell(x) = \ell$} so that $z_i = x'$ by induction.
 Because of $\ell(z_j) = \ell(z_i) = \ell - 1$ induction also gives $z_j = y'$.
 Thus $x = z \sfrak_{ij} = y$.

 \medskip
 (3)
 \emph{Injectivity:}
 Follows immediately from (2).

 \medskip
 (4)
 \emph{Surjectivity:}
 By \cref{lemma:tilting-quiver} it suffices to check for each $x \in \Bcal_+$ with $T_x \in \tilt\Lambda$ that $x \in [1,w_+^2]_L$.
 Firstly, we have \smash{$T_x \geq D(\Lambda) = T_{w_+^2}$} due to \cref{lemma:injective-cogenerator-is-square-of-longest-element}.
 Secondly, there exists $z \in \Bcal_+$ with \smash{$T_{zx} = T_{w_+^2}$} by~(1).
 Finally, we conclude $zx = w_+^2$ with (2), so $x \in [1,w_+^2]_L$.
\end{proof}

\begin{corollary}
 \label{corollary:tilting-modules}
 The tensor products $I_v \otimes_\Lambda I_w$ with $v, w \in \Scal$ are the basic tilting modules for~$\Lambda$.
\end{corollary}

Recall that a $\Lambda$-module $E$ is called \emph{exceptional} if it is self-orthogonal and $\End_\Lambda(E) = K$.

\begin{corollary}
 The modules $e_1 (I_v \otimes_\Lambda I_w)$ with $v, w \in \Scal$ are the exceptional modules for $\Lambda$.
\end{corollary}

\begin{proof}
 Use $\dim_K \End_\Lambda(e_i (I_v \otimes_\Lambda I_w)) = i$ and \cref{corollary:tilting-modules,corollary:self-orthogonals-are-partial-tilting}.
\end{proof}

\begin{corollary}
 \label{corollary:tilting-poset}
 There is a poset isomorphism:
 \[
  \begin{tikzcd}[row sep = 0.1cm]
   [w_-, w_+]_R  \ar[r] & \tilt\Lambda
   \\
   \pad[2em]{x} \ar[r, mapsto] & \pad[0.2em]{T_{x^{-1} w_+}}
  \end{tikzcd}
 \]
\end{corollary}

\begin{proof}
 Use \cref{theorem:tilting-poset} and the fact that $x \mapsto x^{-1} w_+$ defines an anti-isomorphism $[w_-, w_+]_R \to [1, w_+^2]_L$ of posets.
\end{proof}

\begin{remark}
 The poset isomorphism from \cref{corollary:tilting-poset} restricts to an isomorphism $[1, w_+]_R \to \tilt_1\Lambda$.
\end{remark}

Next, we strengthen \cref{theorem:tilting-poset} by describing the \emph{simplicial complex of tilting modules}~$\Sigma(\Lambda)$ combinatorially.
Recall from \cite{Unger:07} that $\Sigma(\Lambda)$ is by definition the abstract simplicial complex whose $r$-dimensional faces are the sets $\{M_0,\ldots,M_r\}$ of isomorphism classes of indecomposable $\Lambda$-modules with the property that $M_0 \oplus \cdots \oplus M_r$ is a direct summand of a tilting module for $\Lambda$.
The vertex set of $\Sigma(\Lambda)$ is by \cref{corollary:self-orthogonals-are-partial-tilting} the set of isomorphism classes of indecomposable self-orthogonal $\Lambda$-modules.

\medskip

We define $\Vcal = \Vcal_n$ as the set \smash{$[1, w_+^2]_L \times \{1,\ldots,n\}$} modulo the equivalence relation $\sim$ generated by $(\sfrak_j x, i) \sim (x, i)$ for $\sfrak_j x >_L x$ and $j \neq i$.
Let $\Sigma = \Sigma_n$ be the $(n-1)$-dimensional abstract simplicial complex with
 \[
  \big\{
   \{ (x, i_0), \ldots, (x, i_r) \} \in \Vcal^{r+1}
   \,|\,
   1 \leq i_0 < \cdots < i_r \leq n
  \big\}
 \]
 as its set of $r$-dimensional faces.

\begin{theorem}
 \label{theorem:simplicial-complex}
 There is an isomorphism $\Sigma \to \Sigma(\Lambda)$ of abstract simplicial complexes given by the assignment $(x, i) \mapsto e_i T_x$.
\end{theorem}

\begin{proof}
 The assignment defines a surjective simplicial map by \cref{lemma:mutation-triangle,theorem:tilting-poset}.
 To prove that it yields an isomorphism, it is enough to check its injectivity.
 For this, assume $e_i T_x \cong U \cong e_i T_y$ for some vertex~$U$ of $\Sigma(\Lambda)$.
 We will show $(x, i) \sim (y, i)$.
 Let $T_x \wedge T_y$ be the meet of $T_x$ and $T_y$ in $\tilt\Lambda$.
 By (1) in the proof of \cref{theorem:tilting-poset} we can choose $x' = \sfrak_{j_\ell} \cdots \sfrak_{j_1}$ with $T_{x' x} = T_x \wedge T_y$.
 Then $T_x \geq T_{x_r} \geq T_x \wedge T_y$ for every $x_r = \sfrak_{j_r} \cdots \sfrak_{j_1} x$ with $0 \leq r \leq \ell$.
 Using \cref{remark:tilting-order,corollary:tilting-lattice} we conclude \smash{$T_x^\perp \supseteq T_{x_r}^\perp \supseteq T_x^\perp \cap T_y^\perp$} and then also \smash{${}^\perp (T_x^\perp) \subseteq {}^\perp (T_{x_r}^\perp)$}.
 Thus
 \[
  U
  \:\in\:
  \add(T_x) \cap \add(T_y)
  \:\subseteq\:
  (T_x^\perp \cap T_y^\perp) \cap {}^\perp (T_x^\perp)
  \:\subseteq\:
  \add(T_{x_r})
 \]
 by \cref{theorem:auslander-reiten}, so $e_i T_{x_r} \cong U$ for all $0 \leq r \leq \ell$.
 It follows $i \not\in \{j_1,\ldots,j_\ell\}$ by \cref{lemma:mutation-triangle} and therefore $(x, i) \sim (x' x, i)$.
 Analogously, there is $y'$ with $T_x \wedge T_y = T_{y' y}$ and $(y, i) \sim (y' y, i)$.
 But then $x' x = y' y$ because of $T_{x' x} = T_x \wedge T_y = T_{y' y}$ and \cref{theorem:tilting-poset}.
 Consequently, $(x, i) \sim (y, i)$.
\end{proof}

The \emph{boundary} $\partial\Delta$ of a pure-dimensional abstract simplicial complex~$\Delta$ is the subcomplex spanned by all its faces of codimension one that are contained in precisely one facet of $\Delta$.
Note that
\[
 \Sigma
 \:=\:
 \partial\Sigma
 \:\overset{.}{\cup}\:
 \{ (1, n) \}
 \:\overset{.}{\cup}\:
 \big\{ F \cup \{ (1, n) \} \,|\, F \in \partial\Sigma \big\}
 \,.
\]
Therefore $\Sigma$ is completely determined by its boundary.

\begin{example}
 The geometric realization of $\partial\Sigma(\Lambda_3)$ looks as follows:
 \[
  \scalebox{0.8}{
   \newcommand{\vertex}[1]{\draw[fill] (#1) circle (0.3mm);}
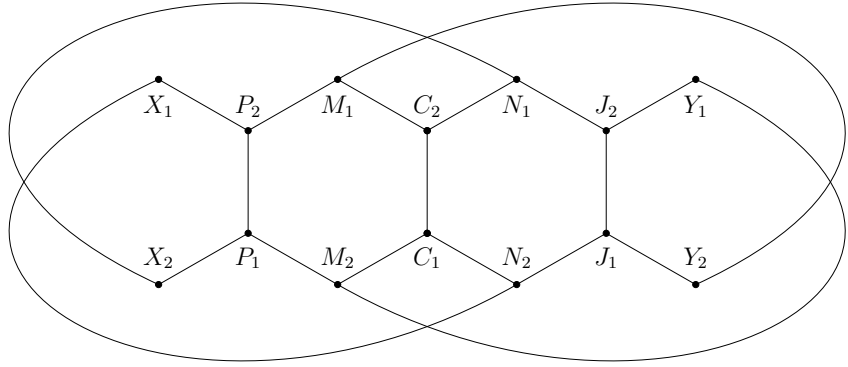
   \begin{tikzpicture}[rotate = -90, scale = 1.7]
    \clip (-1.9,-4.2) rectangle (1.9, 4.2);

    \pgfmathsetmacro{\n}{6}
    \pgfmathsetmacro{\angle}{360/\n}
    \pgfmathsetmacro{\startangle}{180 + \angle}
    \pgfmathsetmacro{\y}{cos(\angle/2)}

    \foreach \i in {1,2,...,\n} {
      \pgfmathsetmacro{\x}{\startangle + \angle*\i}
      \path (0, \y) + (\x:1cm) coordinate (v\i);
      \path (0,-\y) + (\x:1cm) coordinate (w\i);
    }
    \path (v2) + (0, 2*\y) coordinate (x1);
    \path (v5) + (0, 2*\y) coordinate (x2);
    \path (w5) + (0,-2*\y) coordinate (y1);
    \path (w2) + (0,-2*\y) coordinate (y2);

    \foreach \i in {1,2,...,\n} {
      \vertex{v\i};
      \vertex{w\i};
    }
    \vertex{x1};
    \vertex{x2};
    \vertex{y1};
    \vertex{y2};

    \draw (v1) -- (v2) -- (v3) -- (v4) -- (v5) -- (v6);
    \draw (w1) -- (w2) -- (w3) -- (w4) -- (w5) -- (w6) -- (w1);

    \draw (x1) -- (v3);
    \draw (x2) -- (v4);
    \draw (y1) -- (w6);
    \draw (y2) -- (w1);

    \pgfmathsetmacro{\l}{2.5}

    \draw (x2) to[out=  65, in=  60, looseness = \l] (w2);
    \draw (x1) to[out= 115, in= 120, looseness = \l] (w5);
    \draw (y1) to[out= -65, in= -60, looseness = \l] (v2);
    \draw (y2) to[out=-115, in=-120, looseness = \l] (v5);

    \pgfmathsetmacro{\hn}{ 0.74}
    \pgfmathsetmacro{\hs}{-0.74}

    \draw node at (\hn,-3*\y) {$X_2$};
    \draw node at (\hn,-2*\y) {$P_1$};
    \draw node at (\hn,-1*\y) {$M_2$};
    \draw node at (\hn, 0*\y) {$C_1$};
    \draw node at (\hn, 1*\y) {$N_2$};
    \draw node at (\hn, 2*\y) {$J_1$};
    \draw node at (\hn, 3*\y) {$Y_2$};

    \draw node at (\hs,-3*\y) {$X_1$};
    \draw node at (\hs,-2*\y) {$P_2$};
    \draw node at (\hs,-1*\y) {$M_1$};
    \draw node at (\hs, 0*\y) {$C_2$};
    \draw node at (\hs, 1*\y) {$N_1$};
    \draw node at (\hs, 2*\y) {$J_2$};
    \draw node at (\hs, 3*\y) {$Y_1$};
   \end{tikzpicture}
  }
 \]
 Its vertices are the following self-orthogonal modules where $\otimes = \otimes_{\Lambda_3}$:
 \[
  \arraycolsep 5pt
  \begin{array}{l l l l l}
   P_1 = e_1 \Lambda_3
   &
   M_1 = e_1 I_{s_1}
   &
   C_1 = e_1 I_{w_0}
   &
   N_1 = D(I_{s_1} e_1)
   &
   J_1 = D(\Lambda_3 e_1)
   \\
   P_2 = e_2 \Lambda_3
   &
   M_2 = e_2 I_{s_2}
   &
   C_2 = e_2 I_{w_0}
   &
   N_2 = D(I_{s_2} e_2)
   &
   J_2 = D(\Lambda_3 e_2)
  \end{array}
 \]
 \[
  \arraycolsep 40pt
  \begin{array}{l l}
   X_1 = e_1 I_{s_1} \otimes I_{s_1}
   &
   Y_1 = D(I_{s_1} \otimes I_{s_1} e_1)
   \\
   X_2 = e_2 I_{s_2} \otimes I_{s_2}
   &
   Y_2 = D(I_{s_2} \otimes I_{s_2} e_2)
  \end{array}
 \]
\end{example}

\medskip

Returning to the general case, let $p_{n,i}$ be the number of isomorphism classes of modules occurring as $i$-th summand of some tilting module~$T_x$ for~$\Lambda_n$.
By \cref{theorem:simplicial-complex} we have
\[
 p_{n,i}
 \:=\:
 \#\{ (x, i) \in \Vcal_n \}
 \,.
\]

\begin{remark}
 Computations for small $n$ suggest that $p_{n,i}$ is the integer $a_{n+1,i+1}$ in \href{https://oeis.org/A046802}{OEIS:A046802} such that \smash{$p_n := \sum_{1 \leq i \leq n} p_{n,i}$} would be one less than the number of arrangements of an $n$-element set (\href{https://oeis.org/A000522}{OEIS:A000522}):
 \par\vspace{0.2\baselineskip}
 \begin{center}
  \newcolumntype{D}{C{4em}}
  \scalebox{0.8}{
  \begin{tabular}{c||D|D|D|D|D||D}
    $n$ & $p_{n,1}$ & $p_{n,2}$ & $p_{n,3}$ & $p_{n,4}$ & $p_{n,5}$ & $p_n$
    \\ \hline\hline
    1 & 1 & - & - & - & - & 1
    \\ \hline
    2 & 3 & 1 & - & - & - & 4
    \\ \hline
    3 & 7 & 7 & 1 & - & - & 15
    \\ \hline
    4 & 15 & 33 & 15 & 1 & - & 64
    \\ \hline
    5 & 31 & 131 & 131 & 31 & 1 & 325
  \end{tabular}
  }
 \end{center}
 \par\vspace{0.4\baselineskip}
 Since the first summands of the tilting modules $T_x$ are the exceptional modules, this would be in line with \cite{Hille.Ploog:17} who proved that the number of isomorphism classes of exceptional modules for $\Lambda_n$ is $2^n - 1$.
 In fact, checking $p_{n,1} = 2^n - 1$ would give an alternative proof of their result.
\end{remark}

\section{Exceptional sequences}

In this final section we will relate the action of the braid group on full exceptional sequences in $\Dcal(\Lambda)$, which was studied in \cite{Hille.Ploog:17}, to its action on $\DPic(\Lambda)$ by left multiplication (via the map in \cref{proposition:braid-group-to-derived-picard-group}).
For simplicity, we only consider full exceptional sequences of modules.

\medskip

The set of all sequences $(E_1,\ldots,E_n)$ of isomorphism classes of exceptional $\Lambda$-modules with $\Ext^q_\Lambda(E_j, E_i) = 0$ for all $q \geq 0$ and $i < j$ will be denoted by $\exs\Lambda$ in what follows.

\begin{remark}
 For $0 \leq i < n$ and $i^* := n-i$ let $\Delta_{i^*}$ be the \emph{standard module} given by the short exact sequence
\[
 \label{standard-sequence}
 \tag{$\Delta$}
 \begin{tikzcd}
  0
  \ar[r]
  &
  e_i \Lambda
  \ar[r, "\beta_{i+1} \cdot"]
  &
  e_{i+1} \Lambda
  \ar[r]
  &
  \Delta_{i^*}
  \ar[r]
  &
  0
  \,.
 \end{tikzcd}
\]
 Applying the functor $- \Lotimes_\Lambda I_w$ for $w \in \Scal$ yields a short exact sequence
\[
 \begin{tikzcd}
  0
  \ar[r]
  &
  e_i I_w
  \ar[r]
  &
  e_{i+1} I_w
  \ar[r]
  &
  \Ecal_{w,i^*}
  \ar[r]
  &
  0
 \end{tikzcd}
\]
with $\Ecal_{w,i^*} := \Delta_{i^*} \otimes_\Lambda I_w \cong \Delta_{i^*} \Lotimes_\Lambda I_w$ (compare the proof of \cref{lemma:product-of-classical-tilting-is-tilting}).

Then $\Ecal_w := (\Ecal_{w,1},\ldots,\Ecal_{w,n}) \in \exs\Lambda$ because of $(\Delta_1,\ldots,\Delta_n) \in \exs\Lambda$ and \cref{theorem:rickard}.
In view of \cref{theorem:classical-tilting-poset} and \cite[Proposition~3.4]{Hille.Ploog:17} we have a commutative diagram of bijections:
\[
 \begin{tikzcd}
  &
  \Scal
  \ar[ld, "w \mapsto I_w"']
  \ar[rd, "w \mapsto \Ecal_w"]
  \\
  \tilt_1\Lambda
  \ar[rr, yshift = 3pt, "\Phi"]
  &&
  \exs\Lambda
  \ar[ll, yshift = -3pt, "\Psi"]
 \end{tikzcd}
\]

Let us prove that $\Phi$ and $\Psi$ commute with mutation, i.e.\ $\opL_{i^*}(\Ecal_w) = \Ecal_{s_i w}$ whenever $s_i w >_L w$.
By \cref{theorem:rickard,proposition:braid-group-to-derived-picard-group} it suffices to check this for $w = 1$.
Then the left mutation
\[
 \opL_{i^*}(\Ecal_1)
 \:=\:
 (\ldots, \Delta_{i^*-1}, \opL_{\Delta_{i^*}}(\Delta_{i^*+1}), \Delta_{i^*}, \Delta_{i^*+2}, \ldots)
\]
is given by the canonical triangle
\[
 \begin{tikzcd}[column sep = 1.5em]
  \RHom_\Lambda(\Delta_{i^*}, \Delta_{i^*+1}) \otimes_K \Delta_{i^*}
  \ar[r]
  &
  \Delta_{i^*+1}
  \ar[r]
  &
  \opL_{\Delta_{i^*}}(\Delta_{i^*+1})
  \ar[r]
  &
  \cdot
  \:\:
  \,.
 \end{tikzcd}
\]
Taking cohomology, $\opL_{\Delta_{i^*}}(\Delta_{i^*+1})$ is seen to be the middle term of a short exact sequence $0 \to S_i \to \opL_{\Delta_{i^*}}(\Delta_{i^*+1}) \to \Delta_{i^*} \to 0$ that corresponds to a non-zero element of the one-dimensional vector space \smash{$\Ext^1_\Lambda(\Delta_{i^*}, S_i)$}.
An easy calculation that uses $e_i I_i = \rad(e_i \Lambda)$ and $e_j I_i = e_j \Lambda$ for $j \neq i$ now shows $\Ecal_{s_i} = \opL_{i^*}(\Ecal_1)$.
\end{remark}

\begin{remark}
Every tilting module $T = I_v \otimes_\Lambda I_w$ with $v, w \in \Scal$ gives rise to an exceptional sequence $\Ecal = (\Ecal_1, \ldots, \Ecal_n)$ in $\Dcal(\Lambda)$ where
\[
 \begin{tikzcd}[row sep = -0.5ex]
  \Ecal_{i^*}
  \::=\:
  \:\:
  \cdots
  \ar[r]
  &
  0
  \ar[r]
  &
  e_i T
  \ar[r, "\beta_{i+1} \cdot"]
  &
  e_{i+1} T
  \ar[r]
  &
  0
  \ar[r]
  &
  \cdots
  \:\:
  \,.
  \\
  &
  \scalebox{0.7}{$-2$}
  &
  \scalebox{0.7}{$-1$}
  &
  \scalebox{0.7}{$0$}
  &
  \scalebox{0.7}{$1$}
 \end{tikzcd}
\]
This follows by applying $- \Lotimes_\Lambda T$ to the triangle induced by \cref{standard-sequence}.
\end{remark}

\section*{Acknowledgments}

First of all, I would like to thank Julia Sauter for raising the question of whether the tilting modules for the Auslander algebra $\Lambda_n$ can be classified.
I am grateful to Baptiste Rognerud for pointing me to \cite{Digne.Gobet:17}.
I thank both of them and Biao Ma for interesting discussions.

\nocite{OEIS,QPA}
\printbibliography

\end{document}

%% file: packages.tex
\usepackage[utf8]{inputenc}
\usepackage[american]{babel}

\usepackage[
 backend    = biber,
 style      = ieee-alphabetic,
 sorting    = nty,
 giveninits = true,
 isbn       = false,
 doi        = true,
 eprint     = true,
 url        = true,
]
 {biblatex}

\usepackage{hyperref} % include before amsthm
\usepackage{amscd, amsfonts, amsmath, amssymb, amsthm}
\usepackage[nameinlink]{cleveref} % include after hyperref, amsthm, amsmath

\usepackage{csquotes} % include after inputenc
\usepackage{enumitem}
\usepackage{mathtools}
\usepackage{microtype}
\usepackage{tabularx}
\usepackage{tikz, tikz-cd}
\usepackage{xcolor}

%% file: commands.tex
\newcommand{\md}{\operatorname{mod}}
\newcommand{\proj}{\operatorname{proj}}

\newcommand{\gldim}{\operatorname{gl.dim}}
\newcommand{\projdim}{\operatorname{proj.dim}}
\newcommand{\injdim}{\operatorname{inj.dim}}

\newcommand{\Ext}{\operatorname{Ext}}
\newcommand{\Tor}{\operatorname{Tor}}
\newcommand{\Hom}{\operatorname{Hom}}
\newcommand{\End}{\operatorname{End}}

\newcommand{\add}{\operatorname{add}}
\newcommand{\thick}{\operatorname{thick}}

\newcommand{\rad}{\operatorname{rad}}

\newcommand{\Lotimes}{\otimes^\mathbb{L}}
\newcommand{\RHom}{\mathbb{R}\mathrm{Hom}}

\newcommand{\op}{\mathrm{op}}

\newcommand{\tilt}{\operatorname{tilt}}

\newcommand{\fres}{\operatorname{f\text{-}res}}
\newcommand{\fcores}{\operatorname{f\text{-}cores}}
\newcommand{\exs}{\operatorname{exs}}
\newcommand{\DPic}{\operatorname{DPic}}

\newcommand{\opL}{\operatorname{L}}

\newcommand{\Bcal}{\mathcal{B}}

\newcommand{\Dcal}{\mathcal{D}}
\newcommand{\Ecal}{\mathcal{E}}
\newcommand{\Fcal}{\mathcal{F}}

\newcommand{\Kcal}{\mathcal{K}}

\newcommand{\Scal}{\mathcal{S}}

\newcommand{\Vcal}{\mathcal{V}}

\newcommand{\sfrak}{\mathfrak{s}}

\newcommand{\pad}[2][1pt]{\hspace*{#1}#2\hspace*{#1}}

\newtheorem{thm-env}{Theorem}[section]

\newtheorem*{theorem*}{Theorem}
\newtheorem{theorem}[thm-env]{Theorem}
\newtheorem*{proposition*}{Proposition}
\newtheorem{proposition}[thm-env]{Proposition}
\newtheorem*{lemma*}{Lemma}
\newtheorem{lemma}[thm-env]{Lemma}
\newtheorem*{corollary*}{Corollary}
\newtheorem{corollary}[thm-env]{Corollary}
\newtheorem*{definition*}{Definition}

\theoremstyle{definition}

\newtheorem*{remark*}{Remark}
\newtheorem*{example*}{Example}
\newtheorem{remark}[thm-env]{Remark}
\newtheorem{example}[thm-env]{Example}

\crefname{chapter}{Chapter}{Chapters}
\crefname{section}{\S}{\S\S}
\crefname{subsection}{\S}{\S\S}
\crefname{equation}{}{}
\crefname{figure}{Figure}{Figures}
\crefname{table}{Table}{Tables}
\crefname{theorem}{Theorem}{Theorems}
\crefname{proposition}{Proposition}{Propositions}
\crefname{lemma}{Lemma}{Lemmata}
\crefname{corollary}{Corollary}{Corollaries}
\crefname{definition}{Definition}{Definitions}
\crefname{remark}{Remark}{Remarks}
\crefname{example}{Example}{Examples}

\newcolumntype{L}[1]{%
  >{\raggedright\let\newline\\\arraybackslash\hspace{0pt}}m{#1}%
}
\newcolumntype{C}[1]{%
  >{\centering\let\newline\\\arraybackslash\hspace{0pt}}m{#1}%
}
\newcolumntype{R}[1]{%
  >{\raggedleft\let\newline\\\arraybackslash\hspace{0pt}}m{#1}%
}

\definecolor{blueish}{HTML}{001e82}
\definecolor{greenish}{HTML}{0a5a0a}
\definecolor{redish}{HTML}{910019}
\definecolor{turquoisish}{HTML}{056e64}

\hypersetup{
  colorlinks = true,
  linkcolor  = greenish,
  citecolor  = blueish,
  urlcolor   = turquoisish
}

\renewbibmacro*{doi+eprint+url}{%
  \setunit{\adddot\newline}%
  \iftoggle{bbx:doi}
    {\printfield{doi}}
    {}%
  \iffieldundef{doi}
    {}
    {\setunit{\adddot\addspace}}%
  \iftoggle{bbx:url}
    {\usebibmacro{url+urldate}}
    {}%
  \iffieldundef{url}
    {}
    {\setunit{\adddot\addspace}}%
  \iftoggle{bbx:eprint}
    {\usebibmacro{eprint}}
    {}%
}

\DeclareFieldFormat*{url}{%
 \adddot\addspace Online\addcolon\space\url{#1}%
}

\AtEveryBibitem{%
\ifentrytype{book}{
  \clearfield{edition}%
  \clearfield{series}%
  \clearfield{note}%
}{}
\ifentrytype{collection}{
  \clearfield{edition}%
  \clearfield{series}%
  \clearfield{note}%
}{}
\ifentrytype{incollection}{
  \clearfield{edition}%
  \clearfield{series}%
  \clearfield{note}%
}{}
}

\renewbibmacro*{bbx:savehash}{}